\documentclass[11pt]{amsart}
\usepackage{amssymb}

\input{xy}
\xyoption{all}

\newtheorem{theorem}{Theorem}[section]
\newtheorem{proposition}[theorem]{Proposition}
\newtheorem{corollary}[theorem]{Corollary}
\newtheorem{lemma}[theorem]{Lemma}
\newtheorem{conjecture}[theorem]{Conjecture}

\theoremstyle{definition}
\newtheorem{definition}[theorem]{Definition}

\newtheorem{claim}[theorem]{Claim}

\newtheorem{remark}[theorem]{Remark}

\newcommand{\C}{\mathcal C}

\newcommand{\IR}{\mathbb R}
\newcommand{\IN}{\mathbb N}
\newcommand{\II}{\mathbb I}
\newcommand{\w}{\omega}
\newcommand{\cbox}{\operatornamewithlimits{\boxdot}}

\newcommand{\U}{\mathcal U}
\newcommand{\e}{\varepsilon}

\newcommand{\id}{\mathrm{id}}
\newcommand{\V}{\mathcal V}
\newcommand{\pr}{\mathrm{pr}}

\newcommand{\Ra}{\Rightarrow}
\newcommand{\ulim}{\mathrm u\mbox{-}\kern-2pt\varinjlim}
\newcommand{\tlim}{\mathrm t\mbox{-}\kern-2pt\varinjlim}
\newcommand{\lclim}{\mathrm{lc}\mbox{-}\kern-2pt\varinjlim}

\newcommand{\glim}{\mathrm{g}\mbox{-}\kern-2pt\varinjlim}

\newcommand{\lz}{lz}

\title[A topological characterization of $LF$-spaces]{A topological characterization of $LF$-spaces}
\author[T.~Banakh and D.~Repov\v s]{Taras Banakh and Du\v san Repov\v s}


\thanks{This research was supported by Slovenian Research Agency grants P1-0292-0101, J1-2057-0101 and BI-UA/09-10-005.}

\address{Department of Mathematics, Ivan Franko National University of Lviv, and\newline
Instytut Matematyki, Jan Kochanowski University in Kielce, Poland}
\email{t.o.banakh@gmail.com}

\address{Faculty of Mathematics and Physics, and
Faculty of Education,
University of Ljubljana,
P. O. Box 2964,
Ljubljana, Slovenia 1001}
\email{dusan.repovs@guest.arnes.si}

\begin{document}
\begin{abstract} We present a topological characterization of LF-spaces and detect small box-products that are (locally) homeomorphic to LF-spaces.
\end{abstract}

\date{\today}
\maketitle

\section{Introduction}

In this paper we shall present a simple criterion for recognizing topological spaces that are homeomorphic to (open subspaces of) $LF$-spaces. This criterion has been applied in \cite{BMRSY}, \cite{BMSY} and \cite{BY} for detecting topological groups that are homeomorphic to (open subspaces of) LF-spaces.

We recall that an {\em $LF$-space} is the direct limit $\lclim X_n$ of an increasing sequence $$X_0\subset X_1\subset X_2\subset\cdots$$
of Fr\'echet (= locally convex complete linear metric) spaces in the category of locally convex spaces.
The simplest example of a non-metrizable $LF$-space is the inductive limit
$\IR^\infty =\lclim \IR^n$ of the sequence
$$\IR\subset \IR^2\subset\IR^3\subset\cdots$$
of Euclidean spaces, where each space $\IR^n$ is identified with the hyperplane $\IR^n\times\{0\}$ in $\IR^{n+1}$. The space $\IR^\infty$ is topologically isomorphic to the direct sum $\bigoplus_{n\in\w}\IR$ of one-dimensional Fr\'echet spaces in the category of locally convex spaces.

Mankiewicz~\cite{Man}  obtained a topological classification of $LF$-spaces and proved that each $LF$-space is homeomorphic to the direct sum $\bigoplus_{n\in\w}l_2(\kappa_i)$ of Hilbert spaces for some sequence $(\kappa_i)_{i\in\w}$ of cardinals. Here $l_2(\kappa)$ stands for the Hilbert space with orthonormal base of cardinality $\kappa$. In particular, $l_2(n)=\IR^n$ for a finite cardinal $n$. A more precise version of  Mankiewicz's classification says that the spaces
\begin{itemize}
\item $l_2(\kappa)$ for some cardinal $\kappa\ge 0$,
\item $\IR^\infty$,
\item $l_2(\kappa)\times\IR^\infty$ for some $\kappa\ge\w$, and
\item $\bigoplus_{n\in\w}l_2(\kappa_i)$ for a strictly increasing sequence of infinite cardinals $(\kappa_i)_{i\in\w}$
\end{itemize}
are pairwise non-homeomorphic and represent all possible topological types of $LF$-spaces. In particular, each infinite-dimensional separable $LF$-space is homeomorphic to one of the following spaces: $l_2$, $\IR^\infty$ or $l_2\times\IR^\infty$.

The topological characterizations of the LF-spaces $l_2$ and $\IR^\infty$ were given by
Toru\'nczyk \cite{Tor81}, \cite{Tor85} and
Sakai \cite{Sak84}, respectively. These
characterizations belong to
the best achievements of the classical infinite-dimensional topology. In this paper we shall present a topological characterization of other $LF$-spaces, in particular, $l_2\times\IR^\infty$.
First, we recall the Sakai's topological characterization of the LF-space $\IR^\infty$. This characterization is based on the observation that the $LF$-space $\IR^\infty=\lclim\IR^n$ carries the topology of the topological direct limit of the tower $(\IR^n)_{n\in\w}$ of finite-dimensional Euclidean spaces.

By the {\em topological direct limit} $\tlim X_n$ of a tower
$$X_0\subset X_1\subset X_2\subset\cdots
$$of topological spaces we understand the union $X=\bigcup_{n\in\w}X_n$ endowed with the largest topology turning the identity inclusions $X_n\to X$, $n\in\w$, into continuous maps.

\begin{theorem}[Sakai]\label{sakai} A topological space $X$ is homeomorphic to \textup{(}an open subspace of\textup{)} the space $\IR^\infty$ if and only if
\begin{enumerate}
\item $X$ is homeomorphic to the topological direct limit\/ $\tlim X_n$ of a tower $(X_n)_{n\in\w}$ of finite-dimensional metrizable compacta and
\item each embedding $f:B\to X$ of a closed subset $B\subset A$ of a finite-dimensional metrizable compact space $A$ extends to an embedding of \textup{(}an open neighborhood of $B$ in\textup{)} the space $A$ into $X$.
\end{enumerate}
\end{theorem}

Deleting the adjective ``finite-dimensional'' from this characterization, we obtain a characterization of (open subspaces of) the space $Q\times\IR^\infty$ where $Q=[0,1]^\w$ is the Hilbert cube, see \cite{Sak84}.

Replacing the class of finite-dimensional compact metrizable  spaces in Theorem~\ref{sakai} by the class  of Polish spaces, E.~Pentsak~\cite{Pen} obtained a topological characterization of (open subspaces of) the topological direct limit $\tlim (l_2)^n$ of the tower of Hilbert spaces $$l_2\subset l_2\times l_2\subset\dots\subset l_2^n\subset\cdots,$$ where each space $l_2^n$ is identified with the subspace $l_2^n\times\{0\}$ of the Hilbert space $l_2^{n+1}$. However, the topology of the topological direct limit $\tlim l_2^n$ is strictly stronger that the topology of the direct limit $\lclim l_2^n$ of that tower in the category of locally convex spaces. Moreover, $\tlim l_2^n$ is not even homeomorphic to a topological group, see
\cite{Ba98}. In fact, an $LF$-space $X$ is homeomorphic to the topological direct limit of a tower of metrizable spaces if and only if $X$ is either metrizable or is topologically isomorphic to $\IR^\infty$, see \cite{BZ} and \cite{Ba98a}.

This means that topological direct limits cannot be used for describing the topology of non-metrizable $LF$-spaces which are different from $\IR^\infty$. On the other hand, it was discovered in
\cite{BR10} that for any tower $(X_n)_{n\in\w}$ of Fr\'echet spaces the topology of the $LF$-space $X=\lclim X_n$ coincides with the topology of the direct limit $\ulim X_n$ of this tower in the category of uniform spaces!

By the {\em uniform direct limit} $\ulim X_n$ of a tower
$$X_0\subset X_1\subset X_2\subset\cdots$$
of uniform spaces we understand the union $X=\bigcup_{n\in\w}X_n$ endowed with the largest uniformity turning the identity inclusions $X_n\to X$ into uniformly continuous maps.
Each linear topological space $L$ carries the canonical uniformity generated by the entourages $\{(x,y)\in L:x-y\in U\}$ where $U=-U$ runs over all symmetric neighborhoods of the origin of $L$.

For any tower $(X_n)_{n\in\w}$ of Fr\'echet spaces the identity map $\ulim X_n\to\lclim X_n$ is continuous (because each continuous linear operator is uniformly continuous). A less trivial fact established in \cite{BR10} is the continuity of the inverse map $\lclim X_n\to\ulim X_n$. This means that we can identify $LF$-spaces with  uniform direct limits of Fr\'echet spaces and reduce the problem of topological characterization of $LF$-spaces to the problem of recognizing uniform direct limits that are homeomorphic to $LF$-spaces. The answer to this problem will be given in
Theorems~\ref{t:udl-cbox} and \ref{t:sbox} after some definitions.

All spaces considered in this paper are completely regular and all maps are continuous. On the other hand, functions need not be continuous. A {\em pointed space} is a space $X$ with a distinguished point, which will be denoted by $*_X$.

 The {\em small box-product} of a sequence of pointed topological spaces $(X_n)_{n\in\w}$ is the subspace
$$\cbox_{n\in\w}X_n=\{(x_n)_{n\in\w}\in\square_{n\in\w}X_n:\exists m\in\w\;\forall n\ge m\;\;x_n=*_{X_n}\}$$of the box product $\square_{n\in\w}X_n$. The latter space is the Tychonov product $\prod_{n\in\w}X_n$ endowed with the topology generated by the products $\prod_{n\in\w}U_n$ of open subsets $U_n\subset X_n$, $n\in\w$.
For a subset $A\subset\w$ let
$$\cbox_{n\in A}X_n=\{(x_n)_{n\in\w}\in\cbox_{n\in\w}X_n:\{n\in\w:x_n\ne *_{X_n}\}\subset A\}\subset\cbox_{n\in\w}X_n.$$
It follows that $\cbox_{n\in\w}X_n=\bigcup_{n\in\w}\cbox_{i\le n}X_i$.
By Proposition 5.3 of \cite{BR10}, for any sequence $(X_n)_{n\in\w}$ of locally convex linear topological spaces the topology of the small box-product $\cbox_{n\in\w}X_n$ coincides with the topology of the direct sum $\bigoplus_{n\in\w}X_n$ in the category of locally convex linear topological spaces.

For a uniform space $X$ its uniformity will be denoted by $\U_X$. Elements of the uniformity $\U_X$ are called {\em entourages}. The Hausdorff property of $X$ implies that $\cap\U_X=\{(x,x):x\in X\}$. A uniform space $X$ is called {\em metrizable} if its uniformity is generated by a metric.
For a point $a\in X$, a subset $A\subset X$, and an entourage $U\in\U_X$, let $B(a,U)=\{x\in X:(x,a)\in U\}$ and $B(A,U)=\bigcup_{a\in A}B(a,U)$ be the $U$-neighborhoods of $a$ and $A$, respectively. A neighborhood $O(A)$ of $A$ in $X$ is called {\em uniform} if $O(A)$ contains the
$U$-neighborhood $B(A,U)$ for some entourage $U\in\U_X$.

\begin{definition}\label{d:comp} Let $C$ be a pointed topological space. A subset $A$ of a uniform space $X$ is called {\em $C$-complemented} in $X$ if there is a homeomorphism $\gamma:A\times C\to X$ such that
\begin{enumerate}
\item for any neighborhood $V\subset C$ of $*_C$ there is an entourage $U\in\U_X$ such that $B(A,U)\subset\gamma(A\times V)$;
\item for any entourage $U\in\U_X$ there is a neighborhood $V\subset C$ of $*_C$ such that  $\gamma(\{a\}\times V)\subset B(a,U)$ for each $a\in A$.
\end{enumerate}
A subset $A$ of $X$ is called {\em locally $C$-complemented} if for some open neighborhood $V\subset C$ of $*_C$ the set $A$ is $V$-complemented in some open uniform neighborhood $U(A)$ of $A$ in $X$.
\end{definition}

The following theorem shows that often uniform direct limits are (locally) homeomorphic to small box-products. We shall say that a topological space $X$ is {\em locally homeomorphic} to a topological space $Y$ if each point $x\in X$ has an open neighborhood $O_x\subset X$ which is homeomorphic to an open subspace of $Y$.

\begin{theorem}\label{t:udl-cbox} Let $(X_n)_{n\in\w}$ be a tower of uniform spaces such that for every $n\in\w$ the space $X_n$ is (locally) $C_n$-complemented in $X_{n+1}$ for some pointed topological space $C_n$ and $X_n$ is (locally) homeomorphic to the product $X_0\times\cbox_{i<n}C_i$. Then the uniform direct limit $\ulim X_n$ is (locally) homeomorphic to the small box-product $X_0\times\cbox_{n\in\w}C_n$.
\end{theorem}

In light of Theorem~\ref{t:udl-cbox} it is important to recognize small box-products that are (locally) homeomorphic to LF-spaces.
In this respect we have the following:

\begin{conjecture}\label{conj} The small box-product $\cbox_{n\in\w}X_n$ of pointed topological spaces is homeomorphic to (an open subset of) an LF-space if for every $n\in\w$ the finite product $\prod_{i\le n}X_i$ is homeomorphic to (an open subset of) a Hilbert space.
\end{conjecture}

We shall confirm this conjecture under an additional assumption that for infinitely many numbers $n\in\w$ the space $X_n$ is $\lz$-pointed. The definition of an $\lz$-pointed space involves the notion of a strong $Z$-set, well-known in Infinite-Dimensional Topology, see \cite{BM}, \cite[\S 1.4]{BRZ}, \cite[\S2.2]{Chi}.

We recall that a closed subset $A$ of a topological space $X$ is a ({\em strong}) {\em $Z$-set} in $X$ if for any open cover $\U$ of $X$ there is a map $f:X\to X$ such that $f$ is $\U$-near to the identity map $\id_X$ of $X$ and (the closure $\overline{f(X)}$ of) the set $f(X)$ does not intersect $A$. Two maps $f,g:X\to X$ are called {\em $\U$-near} if for each point $x\in X$ the doubleton $\{f(x),g(x)\}$ lies in some set $U\in\U$.
It is clear that each strong $Z$-set is a $Z$-set. The converse is not true, see \cite{BBMW}. However, in Hilbert spaces each $Z$-set is a strong $Z$-set, see \cite{BBMW}, \cite{Tor85}.
A point $x_0$ of a space $X$ will be called a {\em strong $Z$-point} in $X$ is the singleton $\{x_0\}$ is a strong $Z$-set in $X$.

A pointed space $X$ will be called
\begin{itemize}
\item {\em $l$-pointed} if $*_X$ is not isolated and $X$ is locally compact;
\item {\em $z$-pointed} if $*_X$ is a strong $Z$-point in $X$;
\item {\em $\lz$-pointed} if $X$ is $l$-pointed or $z$-pointed.
\end{itemize} For example, each non-trivial Hilbert space $H$ with distinguished point $0$ is an $\lz$-pointed space. More precisely, $H$ is $l$-pointed if $0<\dim(H)<\infty$ and $H$ is $z$-pointed if $\dim(H)=\infty$.

The following theorem (that will be proved in Section~\ref{s:pfsbox})  confirms Conjecture~\ref{conj} for small box-products of $\lz$-pointed spaces.

\begin{theorem}\label{t:sbox} The small box-product $\cbox_{n\in\w}X_n$ of pointed topological spaces $X_n$ is homeomorphic to (an open subset of) an LF-space if for every $n\in\w$ the finite product $\prod_{i\le n}X_i$ is homeomorphic to (an open subset of) a Hilbert space and for infinitely many numbers $n\in\w$ the space $X_n$ is $\lz$-pointed.
\end{theorem}

A subset $A$ of a uniform space $X$ will be called ({\em locally}) {\em $\lz$-complemented} if $A$ is (locally) $C$-complemented in $X$ for some $\lz$-pointed space $C$. By analogy we define (locally) $z$-complemented subsets of uniform spaces.

Theorems~\ref{t:udl-cbox} and \ref{t:sbox} imply the following criterion.

\begin{theorem}\label{t:udl} The uniform direct limit $\ulim X_n$ of a tower of uniform spaces $(X_n)_{n\in\w}$ is
\begin{enumerate}
\item homeomorphic to (an open subset of) an LF-space if for every $n\in\w$ the space $X_n$ is homeomorphic to (an open subset of) a Hilbert space and $X_n$ is $\lz$-complemented in $X_{n+1}$;
\item (locally) homeomorphic to an LF-space if for every $n\in\w$ the space $X_n$ is (locally) homeomorphic to a Hilbert space and $X_n$ is (locally) $\lz$-complemented in $X_{n+1}$.
\item homeomorphic to an LF-space if for every $n\in\w$ the uniform space  $X_n$ is metrizable, is homeomorphic to a Hilbert space and $X_n$ is locally $z$-complemented in $X_{n+1}$.
\end{enumerate}
\end{theorem}

The last statement of this theorem does not follow from Theorems~\ref{t:udl-cbox} and \ref{t:sbox}. It will be proved in a more general context of typical model spaces in Theorem~\ref{t:tudl}.
Because of the lack of the Open Embedding Theorem for LF-manifolds, we distinguish between LF-manifolds and open subspaces of LF-spaces. That is why we included two separate items (1) and (2) in Theorem~\ref{t:udl}. It should be mentioned that  the topological structure of open subspaces of LF-spaces is quite well understood, which cannot be said about LF-manifolds, see \cite{MS}, \cite{MS2}.

Theorem~\ref{t:udl-cbox} will be proved in Section~\ref{s:pfudl}.
In Section~\ref{s:udl} we recall the necessary information on uniform direct limits. In Section~\ref{s:rp} we study reduced products of pointed spaces and prove an important Lemma~\ref{l:rp} on regular homeomorphisms of pairs. In Section~\ref{s:tms} we introduce the notion of a typical model space so that manifolds modeled on such spaces have many common properties with Hilbert manifolds. In Section~\ref{s:comp} we shall prove two lemmas about complemented subsets in metrizable uniform spaces that are homeomorphic to typical model spaces.
 In Section~\ref{s:sbplc} we study small box-product of locally compact spaces and show that for any sequence $(X_i)_{i\in\w}$ of locally compact ANR-spaces the small box-product $Q\times\cbox_{i\in\w}X_i$ is locally homeomorphic to $Q\times\IR^\infty$ where $Q=[0,1]^\w$ is the Hilbert cube. In Section~\ref{s:tsbp} we apply this result to recognize the small box-products that are (locally) homeomorphic to small box-products $\cbox_{n\in\w}E_n$ of typical model spaces. In Section~\ref{s:tudl} we apply the results about small box-products and prove a criterion for recognizing uniform direct limits that are (locally) homeomorphic to small box-products $\cbox_{n\in\w}E_n$ of typical model spaces.

\section{Uniform direct limits}\label{s:udl}

In this section we recall the necessary information on uniform direct limits.
By a {\em tower} of uniform spaces we shall understand an increasing sequence $$X_0\subset X_1\subset X_2\subset\cdots
$$of uniform spaces (so, the uniformity of each space $X_n$ coincides with the uniformity inherited from the uniform space $X_{n+1}$).

The {\em uniform direct limit} $\ulim X_n$ of a tower of uniform spaces $(X_n)_{n\in\w}$ is the union $X=\bigcup_{n\in\w}X_n$ endowed with the largest uniformity making the identity inclusions $X_n\to X$, $n\in\w$,  uniformly continuous. The topology and the uniformity of the uniform direct limits $\ulim X_n$ has been described in \cite{BR10}.

By Proposition 5.4 of \cite{BR10}, for a tower $(X_n)_{n\in\w}$ of locally compact uniform spaces the identity map $\tlim X_n\to\ulim X_n$ is a homeomorphism. This means that the topology of uniform direct limit on $\bigcup_{n\in\w}X_n$ coincides with the topology of topological direct limit.

A map $f:X\to Y$ between uniform spaces is called {\em regular at a subset} $A\subset X$ if for any entourages $U\in\U_Y$ and $V\in\U_X$ there is an entourage $W\in\U_X$ such that for each point $x\in B(A,W)$ there is a point $a\in A$ such that $(x,a)\in V$ and $(f(x),f(a))\in U$.

The following criterion for recognizing continuous maps between uniform direct limits was proved in Theorem~1.6 of \cite{BR10}.

\begin{proposition}\label{p1.1} Let $(X_n)_{n\in\w}$ be a tower of uniform spaces and $X=\ulim X_n$ be its uniform direct limit. A function $f:X\to Y$ from $X$ to a uniform space $Y$ is continuous provided that for every $n\in\IN$ the restriction $f|X_n:X_n\to Y$ is continuous and regular at $X_{n-1}$.
\end{proposition}

Let $X,Y$ be uniform spaces and $X_0\subset X$, $Y_0\subset X$ be subspaces. A homeomorphism of pairs $h:(X,X_0)\to (Y,Y_0)$ is called {\em regular} if $h(X_0)=Y_0$, $h$ is regular at $X_0$ and $h^{-1}$ is regular at $Y_0$.

Proposition~\ref{p1.1} implies a simple criterion for recognizing homeomorphisms between uniform direct limits.

\begin{corollary}\label{c2.1} Let $(X_n)_{n\in\w}$ and $(Y_n)_{n\in\w}$ be towers of uniform spaces. A bijective function $h:\ulim X_n\to \ulim Y_n$ is a homeomorphism if for every $n\in\w$ the restriction $h|X_n$ is a regular homeomorphism of the pairs $(X_{n+1},X_n)$ and $(Y_{n+1},Y_n)$.
\end{corollary}

We shall often use the following fact established in Proposition 5.5 of \cite{BR10}.

\begin{proposition}\label{p2.3} For a sequence $(X_n)_{n\in\w}$ of pointed uniform spaces the identity map $\ulim \cbox_{i\le n}X_i\to\cbox_{n\in\w}X_n$ is a homeomorphism.
\end{proposition}

\section{Proof of Theorem~\ref{t:udl-cbox}}\label{s:pfudl}

We shall divide the proof of Theorem~\ref{t:udl-cbox} into three lemmas.

\begin{lemma}\label{l1n} Let $(X_n)_{n\in\w}$ be a tower of uniform spaces such that for every $n\in\w$ the space $X_n$ is locally $C_n$-complemented in $X_{n+1}$ for some pointed space $C_n$. Then the set $X_0$ has an open neighborhood $U\subset\ulim X_n$ that is homeomorphic to the small box-product $X_0\times \cbox_{n\in\w}W_n$ for some open neighborhoods $W_n\subset C_n$ of the distinguished points $*_{C_n}$.
\end{lemma}

\begin{proof} For every $n\in\w$ the set $X_n$ is locally $C_n$-complemented in $X_{n+1}$. Consequently, for some open neighborhood $W_n\subset C_n$ of $*_{C_n}$ there is an  open embedding $\gamma_n:X_n\times W_n\to X_{n+1}$ such that
\begin{enumerate}
\item[$(\Gamma_1)$] for any neighborhood $V\subset W_n$ of $*_C$ there is an entourage $U\in\U_{X_{n+1}}$ such that $B(X_n,U)\subset \gamma_n(X_n\times V)$;
\item[$(\Gamma_2)$] for each entourage $U\in\U_{X_{n+1}}$ there is a neighborhood $V\subset W_{n}$ of  $*_{C_n}$ such that for any $x\in X_n$ and $c\in V$ we get $\gamma_n(x,c)\in B(x,U)$.
\end{enumerate}
The condition $(\Gamma_2)$ implies that $\gamma_n(x,*_{C_n})=x$ for all $x\in X_n$.

If the set $X_n$ is $C_n$-complemented in $X_{n+1}$, then we shall assume that $W_n=C_n$ and $\gamma_n(X_n\times W_n)=X_{n+1}$.

On each space $C_n$ fix a uniformity that generates the topology of $C_n$ and observe that the map $\gamma_n$ determines a regular homeomorphism of the pairs $(X_n\times W_n,X_n\times\{*_{C_n}\})$ and $(\gamma_n(X_n\times W_n),X_n)$.

Let $U_0=X_0$ and for every $n\in\w$  define an open subset $U_{n+1}\subset X_{n+1}$  by the recursive formula $U_{n+1}=\gamma_n(U_n\times W_n)$.

Let $h_0=\gamma_0:X_0\times W_0\to U_1$. For every $n\in\IN$ define a homeomorphism $h_n:X_0\times\cbox_{i\le n} W_i\to U_{n+1}$ by the recursive formula $h_n(x,c)=\gamma_n(h_{n-1}(x),c)$ where $x\in X_0\times \cbox_{i<n}W_i$ and $c\in W_n$. It follows that $h_n|X_0\times\cbox_{i<n}W_i=h_{n-1}$ and $$h_n:(X_0\times\cbox_{i\le n}W_i,X_0\times \cbox_{i<n}W_i)\to (U_{n+1},U_n)$$is a regular homeomorphism of pairs. By Corollary~\ref{c2.1} and Proposition~\ref{p2.3},  the bijective map $$\mbox{$h=\bigcup_{n\in\w}h_n:\bigcup_{n\in\w}(X_0\times \cbox_{i\le n}W_i)\to \bigcup_{n\in\w}U_n$}$$ is a homeomorphism between the small box-product $X_0\times\cbox_{n\in\w} W_k$ and the uniform direct limit $U=\ulim U_n$.

We claim that $U=\ulim U_n$ is an open subspace of $\ulim X_n$. First we show that the set $U=\bigcup_{n\in\w}U_n$ is open in $X=\ulim X_n$.

Given any point $x\in U$, find $n\in\w$ such that $x\in U_n$. Since the set $U_n$ is open in the uniform space $X_n$, there is an entourage $\e_n\in\U_{X_n}$ such that $B(x,2\e_n)\subset U_n$, where $2\e_n=\e_n\circ\e_n$. Let $B_n=B(x,\e_n)$.

\begin{claim}\label{cl1n} There is a sequence $(\e_k)_{k>n}\in\prod_{k>n}\U_{X_k}$ of entourages such that for every $k>n$ for the set $B_k=B(B_{k-1},\e_k)$ we have the inclusion $B(B_k,\e_k)\subset U_k$.
\end{claim}

\begin{proof} For $k=n$ the inclusion $B(B_n,\e_n)=B(x,2\e_n)\subset U_n$ follows from the choice of $\e_n$. Assume that for some $k\ge n$ we have constructed an entourage $\e_{k}\in\U_{X_k}$ such that $B(B_{k},\e_{k})\subset U_{k}$.

Choose an entourage $\delta_{k+1}\in\U_{X_{k+1}}$ such that $X_k^2\cap 3\delta_{k+1}\subset\e_k$. By the condition $(\Gamma_2)$ there is a neighborhood $V_{k}\subset C_k$ of the distinguished point $*_{C_k}$ such that for every $c\in V_k$ and $x_k\in X_k$ we get $\gamma_k(x_k,c)\in B(x_k,\delta_{k+1})$. By the condition $(\Gamma_1)$, there is  an entourage $\e_{k+1}\in\U_{X_{k+1}}$ such that $B(X_k,2\e_{k+1})\subset \gamma_k(X_k\times V_k)$ and $\e_{k+1}\subset\delta_{k+1}$. Let us check that the entourage $\e_{k+1}$ satisfies our requirements.

Consider the set $B_{k+1}=B(B_k,\e_{k+1})$ and its $\e_{k+1}$-neighborhood $B(B_{k+1},\e_{k+1})=B(B_{k},2\e_{k+1})$. We need to show that $B(B_k,2\e_{k+1})\subset U_{k+1}$. Take any point $y\in B(B_k,2\e_{k+1})\subset B(X_k,2\e_{k+1})\subset\gamma_k(X_k\times V_k)$
and find a pair $(x_k,c)\in X_k\times V_k$ such that $y=\gamma_k(x_k,c)$.
The choice of the neighborhood $V_k$ guarantees that $(y,x_k)\in \delta_{k+1}$. Then $x_k\in X_k\cap B(y,\delta_{k+1})\subset X_k\cap (B_k,\delta_{k+1}\circ 2\e_{k+1})\subset X_k\cap B(B_k,3\delta_{k+1})=B(B_k,X_k^2\cap 3\delta_{k+1})\subset B(B_k,\e_k)\subset U_k$ and hence $y=\gamma_k(x_k,c)\in \gamma_k(U_k\times W_k)=U_{k+1}$.
\end{proof}

Theorem~1.1 of \cite{BR10} guarantees that the union $B_\infty=\bigcup_{k\ge n}B_k$ is a neighborhood of the point $x$ is $\ulim X_n$. Since $B_\infty=\bigcup_{k\ge n}B_k\subset\bigcup_{k\ge n}U_k=U$, we see that the point $x$ lies in the interior of $U$ and hence the set $U$ is open in $\ulim X_n$. Claim~\ref{cl1n} and the description of the topology of uniform direct limits given in \cite[1.1]{BR10} guarantee that the topology of the uniform direct limit on $\ulim  U_n$ coincides with the subspace topology inherited from $\ulim X_n$.
This completes the proof of the lemma.
\end{proof}

If each subset $X_n$ is $C_n$-complemented in $X_{n+1}$, then $W_n=C_n$ and $\gamma_n(X_n\times C_n)=X_{n+1}$ for all $n\in\w$. By induction we can prove that $U_n=X_n$ for all $n\in\w$ and hence the uniform direct limit $\ulim X_n=\ulim U_n$ is homeomorphic to $X_0\times \cbox_{n\in\w}C_n$.
This argument yields the following version of Lemma~\ref{l1n}.

\begin{lemma}\label{l2n} Let  $(X_n)_{n\in\w}$ be a tower of uniform spaces such that for every $n\in\w$ the space $X_n$ is  $C_n$-complemented in $X_{n+1}$ for some pointed topological space $C_n$. Then the uniform direct limit  $X=\ulim X_n$ is homeomorphic to the small box-product $X_0\times \cbox_{n\in\w}C_n$.
\end{lemma}

Our final lemma completes the proof of Theorem~\ref{t:udl-cbox}.

\begin{lemma} Let $(X_n)_{n\in\w}$ be a tower of uniform spaces such that for every $n\in\w$ the space $X_n$ is locally $C_n$-complemented in $X_{n+1}$ for some pointed space $C_n$ and is locally homeomorphic to $X_0\times\prod_{i<n}C_i$. Then the space $\ulim X_n$ is locally homeomorphic to  $X_0\times \cbox_{n\in\w}C_n$.
\end{lemma}

\begin{proof} Given any point $x\in \ulim X_n$, find a number $n\in\w$ with $x\in X_n$. By Lemma~\ref{l1n}, the point $x$ has an open neighborhood  $O(x)$ that is homeomorphic to the small box-product $X_n\times \cbox_{i\ge n}W_i$ for some open neighborhoods $W_i\subset C_i$ of the distinguished points $*_{C_i}$. Since the space $X_n$ is locally homeomorphic to $X_0\times\prod_{i<n}C_i$, we conclude that $X_n\times\cbox_{i\ge n}W_i$ is locally homeomorphic to $X_0\times\prod_{i<n}C_i\times \cbox_{i\ge n}C_i=X_0\times\cbox_{i\in\w}C_i$. Consequently, $x$ has an open neighborhood, homeomorphic to an open subset of $X_0\times\cbox_{i\in\w}C_i$ witnessing that $X$ is locally homeomorphic to $X_0\times\cbox_{i\in\w}C_i$.
\end{proof}

\section{Reduced products}\label{s:rp}

In this section we collect the necessary information on reduced products of pointed spaces. This information will be used in the proofs of Theorems~\ref{t:sbox} and \ref{t:tsbox}.

For a pointed space $C$ with distinguished point $*_C$ we denote by $C^\circ=C\setminus\{*_C\}$ its complement in $C$.

The {\em reduced product} $C\rtimes E$ of a pointed topological space $C$ and a topological space $E$ is the space $$(C^\circ\times E)\cup\{*_C\}$$
endowed with the smallest topology such that the identity inclusion $C^\circ\times E\to C\rtimes E$ is an open topological embedding and the natural projection $\pr:C\rtimes E\to C$ is continuous. The reduced product $C\rtimes E$ is a pointed space with the distinguished point $*_{C\rtimes E}=*_C$.

If $C$ and $E$ are uniform spaces, then their reduced product $C\rtimes E$ carries the smallest uniformity such that the projection $C\rtimes E\to C$ is uniformly continuous and for every closed subspace $F\subset C^\circ$ of $C$ the embedding $F\times E\to C\rtimes E$ is a uniform homeomorphism.

A map $f:X\to Y$ between topological spaces is called a {\em near homeomorphism} if for any open cover $\U$ of $Y$ there is a homeomorphism $h:X\to Y$ that is {\em $\U$-near} to $f$.

\begin{lemma}\label{l:rp} Let $M,N,E$ be metric spaces and $C$ be a pointed metric space. If the projection $\pr:M\times C^\circ\times E\to M\times C^\circ$, $\pr:(x,y,z)\mapsto (x,y)$, is a near homeomorphism, then for any homeomorphism $f:M\to N$ there is a regular homeomorphism of pairs
$$\bar f:(M\times (C\rtimes E),M\times\{*_{C\rtimes E}\})\to (N\times C,N\times\{*_{C}\})$$such that $\bar f(x,*_{C\rtimes E})=(f(x),*_C)$ for all $x\in M$.
\end{lemma}

\begin{proof} Let $d_N$ and $d_C$ be the metrics of the metric spaces $N$ and $C$, respectively. These metrics determine the metric
$$d\big((x,y),(x',y')\big)=\max\{d_N(x,x'),d_C(y,y')\}$$on the product $N\times C$. Since the projection $\pr:M\times C^\circ\times E\to M\times C^\circ$ is a near homeomorphism and $M$ is homeomorphic to $N$, there exists a homeomorphism
$h:N\times C^\circ\times E\to N\times C^\circ$ such that
$$d(h(x,c,e),\pr(x,c,e))\le \frac13d_C(c,*_C)\mbox{ \ for all \ $(x,c,e)\in N\times C^\circ\times E.$}$$
Extend $h$ to a homeomorphism $\bar h:N\times(C\rtimes E)\to N\times C$ by letting $\bar h|N\times C^\circ=h$ and $\bar h|N\times\{*_C\}=\id$.

The homeomorphism $f:M\to N$ induces a homeomorphism $f\times\id:M\times (C\rtimes E)\to N\times (C\rtimes E)$, $f\times\id:(x,y)\mapsto (f(x),y)$.

Now consider the homeomorphism $\bar f=\bar h\circ (f\times\id):M\times (C\rtimes E)\to N\times C$ and observe that for each $(x,c,e)\in M\times C^\circ \times E\subset M\times (C\rtimes E)$ we get
$$\frac23 d_C(c,*_C)\le d\big(\bar h(f(x),c,e),N\times\{*_C\}\big)\le d\big(\bar h(f(x),c,e),(f(x),*_C)\big)\le \frac43d_C(c,*_C),$$
which implies that $$\bar h:(M\times (C\rtimes E),M\times\{*_{C\rtimes E}\})\to (N\times C,N\times\{*_C\})$$ is a regular homeomorphism of pairs.
\end{proof}

\section{Typical model spaces}\label{s:tms}

In fact, Theorem \ref{t:sbox} holds in a more general setting with LF-spaces replaced by small box-products of typical model spaces.

\begin{definition}\label{d:tms} A pointed topological space $E$ is called a {\em typical model space} if
\begin{enumerate}
\item $E$ is a topologically homogeneous absolute retract containing a topological copy of the Hilbert cube $Q=[0,1]^\w$;
\item For any neighborhood $U\subset E$ of $*_E$ there are neighborhoods $V,W\subset U$ of $*_E$ such that $W$ and $E\setminus V$ are homeomorphic to $E$ and the boundary $\partial V$ of $V$ is a retract of $\overline{V}$ and a $Z$-set in $E\setminus V$;
\item each contractible $E$-manifold is homeomorphic to $E$;
\item each connected $E$-manifold $M$ is homeomorphic to an open subset of $E$;
\item any homeomorphism $h:A\to B$ between $Z$-sets $A,B\subset E$ extends to a homeomorphism $\bar h:E\to E$ of $E$;
\item for any $E$-manifold $M$ the projection $E\times M\to M$ is a near homeomorphism;
\item for any retract $X$ of an open subset of $E$ the product $X\times E$ is homeomorphic to an open subset of $E$;
\item for any retract $X$ of an $E$-manifold and a strong $Z$-point $*_X\in X$ the reduced product $X\rtimes E$ is an $E$-manifold, homeomorphic to   $X\times E$.
\end{enumerate}
\end{definition}

By an {\em ANR-space} we understand a metrizable space $X$, which is a neighborhood retract in each metric space that contains  $X$ as a closed subspace.

The theory of Hilbert manifolds developed in \cite{HS}, \cite[\S IX.7]{BP}, \cite{Tor74}, \cite{Tor81}, \cite{Tor85} yields the following theorem.

\begin{theorem} Any infinite-dimensional Hilbert space is a typical model space.
\end{theorem}

\begin{remark} Many incomplete typical model spaces can be found among absorbing and coabsorbing spaces, see \cite{BM} and \cite{BRZ}.
\end{remark}

We finish this short section by a lemma that will be used in the proof of Theorem~\ref{t:tsbox}.

\begin{lemma}\label{l:Z} Let $X$ be a pointed ANR-space and $Y$ be a pointed topological space. If $*_X$ is a strong $Z$-point in $X$, then $(*_X,*_Y)$ is a strong $Z$-point in $X\times Y$.
\end{lemma}

\begin{proof} Given an open cover $\U$ of $X\times Y$, find a set $U\in\U$ that contains the point $(*_X,*_Y)$. Choose open sets  $U_X\subset X$ and $U_Y\subset Y$ such that $(*_X,*_Y)\in U_X\times U_Y\subset U$ and find a neighborhood $V_X\subset X$ of $*_X$ such that $\overline{V}_X\subset U_X$.
Since the space $Y$ is completely regular, there is a continuous function $\lambda_Y:Y\to[0,1]$ such that $\lambda_Y^{-1}(0)\supset Y\setminus U_Y$ and $\lambda_Y^{-1}(1)$ is a neighborhood of $*_Y$.
By the same reason, there is a continuous function $\lambda_X:X\to[0,1]$ such that $\lambda_X^{-1}(0)\supset X\setminus V_X$ and $\lambda^{-1}_X(1)$ is a neighborhood of $*_X$.
 Let $\Lambda_X$ be the interior of $\lambda_X^{-1}(1)$ in $X$ and $W_X$ be a neighborhood of $*_X$ in $X$ such that $\overline{W}_X\subset\Lambda_X$.
Consider the open cover $\V=\{\Lambda_X,U_X\setminus \overline{W}_X,X\setminus \overline{V}_X\}$ of $X$. Since $X$ is an ANR, there is an open cover $\mathcal{W}$ of $X$ such that any two $\mathcal W$-near maps into $X$ are $\V$-homotopic. Since $*_X$ is a strong $Z$-point, there is a map $f:X\to X$ such that $f$ is $\mathcal W$-near to $\id_X$ and $*_X\notin\overline{f(X)}$. By the choice of the cover $\mathcal W$, the map $f$ is $\V$-homotopic to $\id_X$. Consequently, there is a homotopy $h:X\times [0,1]\to X$ such that for every $x\in X$ we get
$h(x,0)=x$, $h(x,1)=f(x)$ and $h(\{x\}\times[0,1])\subset V_x$ for some $V_x\in \V$.

Now consider the function $\lambda:X\times Y\to[0,1]$, $\lambda:(x,y)\mapsto\lambda_X(x)\cdot\lambda_Y(y)$, and the map $$g:X\times Y\to X\times Y,\;\;g:(x,y)\mapsto \big(h(x,\lambda(x,y)),y\big).$$

\begin{claim} The map $g$ is $\U$-near to $\id_{X\times Y}$.
\end{claim}

\begin{proof}
Take any pair $(x,y)\in X\times Y$. If $(x,y)\notin \overline{V}_X\times U_Y$, then $$g(x,y)=\big(h(x,\lambda(x,y)),y\big)=\big(h(x,0),y\big)=(x,y)$$and hence the singleton $\{g(x,y),(x,y)\}$ lies in some element of the cover $\U$. Next, assume that $(x,y)\in \overline{V}_X\times U_Y$.
Since $h$ is a $\V$-homotopy, and $h(x,0)=x\notin X\setminus\overline{V}_X$, $h(\{x\}\times[0,1])\subset U_X$. Then $$g(x,y)=\big(h(x,\lambda(x,y)),y)\in U_X\times U_Y\subset U\in\U$$ and $(x,y)\in\overline{V}_X\times U_Y\subset U_X\times U_Y\in U\in\U$.
\end{proof}

Consider the neighborhood $W_X'=W_X\cap (X\setminus \overline{f(X)})$ of $*_X$ and the neighborhood $W=W_X'\times\lambda_Y^{-1}(1)$ of $(*_X,*_Y)$ in $X\times Y$.

\begin{claim} The neighborhood $W$ does not intersect $g(X\times Y)$.
\end{claim}

\begin{proof} Fix any point $(x,y)\in X\times Y$. If $y\notin\lambda^{-1}(1)$, then $g(x,y)\in X\times\{y\}\subset X\times (Y\setminus \lambda^{-1}(1))\subset (X\times Y)\setminus W$.

So, we assume that $y\in\lambda^{-1}(1)$. If $x\notin\overline{V}_X$, then $g(x,y)=(x,y)\notin W$. If $x=h(x,0)\in \overline{V}_X\setminus \Lambda_X$, then $h(\{x\}\times[0,1])\subset U_X\setminus \overline{W_X}$ as $h$ is a $\V$-homotopy. In this case $g(x,y)\in (X\setminus \overline{W}_X)\times Y\subset X\times Y\setminus W$.

If $x\in\Lambda_X$, then $\lambda(x,y)=\lambda_X(x)\cdot\lambda_Y(y)=1$ and  $g(x,y)=(f(x),y)\notin W$.
\end{proof}
Thus the map $g$ witnesses that $(*_X,*_Y)$ is a strong $Z$-point in $X\times Y$.
\end{proof}

\section{Complemented subsets in typical model spaces}\label{s:comp}

In this section we prove two useful lemmas about complemented subsets in metrizable uniform spaces that are locally homeomorphic to typical model spaces.

\begin{lemma}\label{l:rcomp} Let $C$ be a pointed topological space and $A$ be a $C$-complemented subset of a metrizable uniform space $X$. If $X$ is an $E$-manifold for some typical model space $E$, then $A$ is $C\rtimes E$-complemented in $X$.
\end{lemma}

\begin{proof}
Since $A$ is $C$-complemented in $X$, there is a homeomorphism $\gamma:A\times C\to X$ satisfying the conditions (1) and (2) of Definition~\ref{d:comp}. By our hypothesis, the uniform space $X$ is metrizable and hence
its uniformity is generated by some bounded metric $\rho_X$.
Since $X$ is homeomorphic to the product $A\times C$, the space $C$ is metrizable, so we can choose a metric $\rho_C\le 1$ that generates the topology of $C$. Denote by $\U_C$ the uniformity on $C$ generated by the metric $\rho_C$.

By induction construct a sequence of entourages $(U_n)_{n\in\w}\in\U_X^\w$ and $(V_n)_{n\in\w}\in\U_C^\w$ such that $U_0=X\times X$, $V_0=C\times C$ and for every $n\in\IN$ the following conditions are satisfied:
\begin{enumerate}
\item $B(A,U_n)\subset \gamma(A\times B(*_C,V_{n-1}))$;
\item $\gamma(a,c)\in B(a,U_n)$ for each $a\in A$ and $c\in B(*_C,V_{n})$;
\item $U_n\circ U_n\circ U_n\subset U_{n-1}=U_{n-1}^{-1}\subset \{(x,x')\in X^2:\rho_X(x,x')\le2^{-n+1}\}$;
\item $V_n\circ V_n\circ V_n\subset V_{n-1}=V_{n-1}^{-1}\subset \{(c,c')\in C^2:\rho_C(c,c')\le 2^{-n+1}\}$.
\end{enumerate}
By Theorem~\cite[8.1.10]{En}, there are pseudometrics $d_X$ and $d_C$ on $X$ and $C$, respectively, such that for every $n\in\w$
\begin{enumerate}
\item[(5)] $\{(x,x')\in X^2:d_X(x,x')<2^{-n}\}\subset U_n\subset\{(x,x')\in X^2:d_X(x,x')\le 2^{-n}\}$ and
\item[(6)] $\{(c,c')\in C^2:d_C(c,c')<2^{-n}\}\subset V_n\subset\{(c,c')\in C^2:d_C(c,c')\le 2^{-n}\}$.
\end{enumerate}
 The conditions (3), (5) and (4), (6) imply that $d_X$ and $d_C$ are metrics generating the uniformities of the corresponding spaces.

For $\e>0$ let $B_X(A,\e)=\{x\in X:d_X(x,A)<\e\}$ and $B_C(*_C,\e)=\{c\in C:d_C(c,*_C)<\e\}$.
The conditions (1), (2) and (5), (6) imply that for every $\e>0$ the following two conditions are satisfied:
\begin{enumerate}
\item[(7)] $B(A,\e)\subset \gamma(A\times B(*_C,4\e))$;
\item[(8)] $\gamma(a,c)\in B(a,4\e)$ for each $a\in A$ and $c\in B(*_C,\e)$;
\end{enumerate}

The metrics $d_X$ and $d_C$ generate the metric
$$d_{AC}((a,c),(a',c'))=\max\{d_X(a,a'),d_C(c,c')\}$$on the product $A\times C$.

Let $C^\circ=C\setminus\{*_C\}$. The space $A\times C^\circ$ is an $E$-manifold (being homeomorphic to the open subset $\gamma(A\times C^\circ)$ of the $E$-manifold $X$).
Consequently, the projection $p:A\times C^\circ\times E\to A\times C^\circ$ is a near homeomorphism and we can choose a homeomorphism $h:A\times C^\circ \times E\to A\times C^\circ$ such that
$$d_{AC}(h(a,c,e),p(a,c,e))\le \frac13d_C(c,*_C) \mbox{ \ and \  }d_X(\gamma\circ h(a,c,e),\gamma\circ p(a,c,e))\le\frac 12d_C(c,*_C)$$for all $(a,c,e)\in A\times C^\circ \times E$.

Extend $h$ to a homeomorphism $\bar h:A\times C\rtimes E\to A\times C$ letting $\bar h|A\times C^\circ \times E=h$ and $\bar h|A\times\{*_C\}=\id$.

It can be shown that the homeomorphism $\tilde \gamma=\gamma\circ\bar h:A\times C\rtimes E\to X$ witnesses that $A$ is $C\rtimes E$-complemented in $X$.
\end{proof}

Our second lemma reduces the local $z$-complementability in uniform spaces homeomorphic to typical model spaces to the $E$-complementability.

 \begin{lemma}\label{l:zcomp} Let $A$ be a retract of a metrizable uniform space $X$ such that $X$ is homeomorphic to some typical model space $E$. If $A$ is locally $z$-complemented in $X$, then $A$ is $E$-complemented in $X$.
\end{lemma}

\begin{proof} Assuming that $A$ is locally $z$-complemented in $X$, find a $z$-pointed space $C$, an open neighborhood $V\subset C$ of $*_C$ and an open uniform neighborhood $U\subset X$ of $X$ such that $A$ is $V$-complemented in $U$. It follows that $A\times V$ is homeomorphic to $U$ and hence $A\times V$ and $V$ are ANRs. The ANR-property of $V$ can be used to show that the distinguished point $*_C$ is a strong $Z$-point not only in $C$ but also in $V$. So, we lose no generality assuming that $V=C$. By Lemma~\ref{l:rcomp}, the set $A$ is $C\rtimes E$-complemented in $U$. By the condition (8) of Definition~\ref{d:tms}, the reduced product $C\rtimes E$ is an $E$-manifold. Consequently, the distinguished point $*_{C\rtimes E}$ has a neighborhood in $C\rtimes E$, homeomorphic an open neighborhood $U\subset E$ of the distinguished point $*_E$ of $E$.
It follows that the set $A$ is locally $U$-complemented in $X$. Hence, we can find an open embedding $\gamma:A\times U\to X$ satisfying the condition (1), (2) of Definition~\ref{d:comp}.
By the condition (2) of Definition~\ref{d:tms} the distinguished point $*_E$ of $U$ has a neighborhood $V\subset E$ such that $\overline{V}\subset U$, the complement $E\setminus V$ is homeomorphic to $E$ and the boundary $\partial V=\overline{V}\setminus V$ of $V$ is a retract of $\overline{V}$ and a $Z$-set in $E\setminus V$.

Since $A$ is a retract of the space $X$ and the spaces $X$ and $E\setminus V$ are homeomorphic to $E$, the product $A\times (E\setminus V)$ is homeomorphic to $E$, being a contractible $E$-manifold. Since $\partial V$ is a $Z$-set in $E\setminus V$, the product $A\times\partial V$ is a $Z$-set in $A\times(E\setminus V)$.

We claim that the complement $M=X\setminus\gamma(A\times V)$ is homeomorphic to $E$ and $\gamma(A\times\partial V)$ is a $Z$-set in $M$.
First we show that $M$ is contractible. Since $\partial V$ is a retract of $\overline{V}$, $A\times\partial V$ is a retract of $A\times\overline{V}$ and hence $\gamma(A\times\partial V)$ is a retract of $\gamma(A\times\overline{V})$. Then $M$ is a retract of $X=M\cup\gamma(A\times\overline{V})$. Since $X$ is contractible, so is its retract $M$. To see that $M$ is an $E$-manifold, observe that $M$ is the union of two open subsets $X\setminus\gamma(A\times \overline{V})$ and $\gamma(A\times (U\setminus V))$, the first of which is open in the $E$-manifold $X$ while the second is a topological copy of the $E$-manifold $A\times (U\setminus V)\subset A\times(E\setminus V)$. Being a contractible $E$-manifold, the space $M$ is homeomorphic to $E$.

Since $\partial V$ is a $Z$-set in $E\setminus V$, it is a $Z$-set in $U\setminus V$. Then $A\times\partial V$ is a $Z$-set in $A\times(U\setminus V)$ and $\gamma(A\times \partial V)$ is a $Z$-set in the open subset $\gamma(A\times(U\setminus V))$ of $M$ and hence a $Z$-set in $M$.

Then $\gamma|A\times\partial V$ is a homeomorphism between the $Z$-sets $A\times\partial V$ and $\gamma(A\times\partial V)$ is the spaces $A\times(E\setminus V)$ and $M$ which are homeomorphic to $E$. By the condition (5) of Definition~\ref{d:tms}, there is a homeomorphism $h:A\times (E\setminus V)\to M$ such that $h(x)=\gamma(x)$ for all $x\in A\times\partial V$. Extend the homeomorphism $h$ to a homeomorphism $\tilde\gamma:A\times E\to X$ letting  $\tilde \gamma|A\times \overline{V}=\gamma|A\times \overline{V}$ and $\tilde \gamma|A\times (E\setminus V)=h$. The homeomorphism $\tilde \gamma$ witnesses that the set $A$ is $E$-complemented in $X$.
\end{proof}

\section{Small box-products of locally compact spaces}\label{s:sbplc}

In this section we study the topological structure of small box products of pointed locally compact ANR-spaces. By $Q=[0,1]^\w$ we denote the Hilbert cube.  A pointed space $X$ is called {\em non-isolated} if its distinguished point $*_X$ is not isolated in $X$. By a {\em polyhedron} we understand a topological space, homeomorphic to the geometric realization of some simplicial complex.

The following theorem is the main result of this section.

\begin{theorem}\label{t:sboxlc} For any sequence $(X_n)_{n\in\IN}$ of non-isolated pointed
locally compact ANR-spaces the small box-product $Q\times\cbox_{n\in\IN}X_n$ is homeomorphic to the product $K\times Q\times\IR^\infty$ for some locally compact polyhedron $K$. If each space $X_n$, $n\in\IN$, is contractible, then the small box-product $Q\times \cbox_{n\in\IN}X_n$ is homeomorphic to $Q\times\IR^\infty$.
\end{theorem}

\begin{proof} Put $X_0=Q$ and instead of the product $Q\times\cbox_{n\in\IN}X_n$ consider the small box-product $\cbox_{n\in\w}X_n$. In order to prove that $\cbox_{n\in\w}X_n$ is a $Q\times\IR^\infty$-manifold, we shall apply Sakai's characterization \cite{Sak84} of open subspaces of $Q\times\IR^\infty$, mentioned in the Introduction.

We can assume that each space $X_n$ carries a uniformity that generates its topology. Since each ANR-space is locally connected, it suffices to prove the theorem in the case of connected locally compact ANR-spaces $X_n$.
 In this case each space $X_n$ is  $\sigma$-compact and so is each finite product $\cbox_{i<n}X_i$. Then
$\big(\cbox_{i\le n}X_i\big)$ is a tower of locally compact $\sigma$-compact uniform spaces. By Proposition 5.4 and 5.5 of \cite{BR10} the identity maps
$$\tlim \cbox_{i\le n}X_i\to\ulim \cbox_{i\le n}X_i\to \cbox_{n\in\w}X_n$$are homeomorphisms.

Taking into account that each finite product $\cbox_{i\le n}X_i$ is locally compact and $\sigma$-compact, we can show that the topological direct limit $\tlim \cbox_{i\le n}X_i$ is a $k_\w$-space, which means that it can be written as a topological direct limit of a tower of compact metrizable spaces. Now the Sakai's characterization \cite{Sak84} will imply that $\cbox_{n\in\w}X_i$ is homeomorphic to an open subset of the space $Q\times\IR^\infty$ as soon as we show that each embedding $f:B\to \cbox_{n\in\w}X_i$ defined on a closed subset $B$ of a compact metrizable space $A$ extends to an embedding $\bar f:O(B)\to \cbox_{n\in\w}X_i$ of some neighborhood $O(B)$ of $B$ in $A$.

Since $f(B)$ is a compact subset of the topological direct limit $\tlim \cbox_{i\le n}X_i$, there is $n\in\IN$ such that $f(B)\subset\cbox_{i< n}X_i$. Since $\cbox_{i< n}X_i$ is an ANR, the map $f:B\to \cbox_{i< n}X_i$ admits a continuous extension $\bar f:O(B)\to \cbox_{i< n}X_i$ to some closed neighborhood $O(B)$ of $B$ in $A$.

By the ANR-Theorem for Q-manifolds \cite[44.1]{Chap}, the product $\cbox_{i< n}X_i=Q\times\cbox_{1{\le} i{<}n}X_i$ is a $Q$-manifold and so is the product $[0,1]\times\cbox_{i< n}X_i$. Identify $\cbox_{i< n}X_i$ with the $Z$-set $\{0\}\times\cbox_{i< n}X_i$ in $[0,1]\times\cbox_{i\le n}X_i$. By Theorem~18.2 of \cite{Chap}, the map $\bar f:O(B)\to \cbox_{i< n}X_i$ can be approximated by an embedding $\tilde f:O(B)\to [0,1]\times\cbox_{i< n}X_i$ such that $\tilde f|B=f$.

Since  $X_n$ is a non-isolated pointed space, there is an embedding $\gamma:[0,1]\to X_n$ such that $\gamma(0)=*_X$. The embedding $\gamma$ induces the embedding $$\tilde\gamma:[0,1]\times\cbox_{i<n}X_i\to\cbox_{i\le n}X_i,\;\;\tilde\gamma:(t,\vec x)\mapsto (\vec x,\gamma(t)).$$
Then $g=\tilde\gamma\circ\tilde f:O(B)\to\cbox_{i\le n}X_i\subset\cbox_{i\in\w}X_i$ is a required embedding that extends the embedding $f$. By Sakai's characterization of $Q\times\IR^\infty$-manifolds \cite{Sak84}, $\cbox_{i\in\w}X_i$ is a $Q\times\IR^\infty$-manifold and by the Triangulation Theorem \cite{Sak84} for $Q\times\IR^\infty$-manifolds, the $Q\times\IR^\infty$-manifold $\cbox_{i\in\w}X_i$ is homeomorphic to $K\times Q\times\IR^\infty$ for some locally compact polyhedron $K$.
\smallskip

If each space $X_n$ is contractible, the product $\cbox_{i<n}X_i$ is an absolute retract. In this case,  we can assume that $O(B)=A$ and then the embedding $f:B\to \cbox_{i\in\w}X_i$ extends to an embedding $\bar f:A\to\cbox_{i\in\w}X_i$. By Sakai's characterization \cite{Sak84} of the space $Q\times\IR^\infty$, the $k_\omega$-space $\cbox_{i\in\w}X_i$ is homeomorphic to $Q\times\IR^\infty$.
\end{proof}

\section{The topological structure of some small box-products}\label{s:tsbp}

In this section we prove a ``typical'' version of Theorem~\ref{t:sbox}.

\begin{theorem}\label{t:tsbox} Let $(X_n)_{n\in\w}$ be a sequence of  pointed topological spaces such that for every $n\in\w$ the finite product $\prod_{i\le n}X_n$ is homeomorphic to (an open subspace of) some typical model space $E_n$. Assume that for infinitely many numbers $n\in\w$ the space $X_n$ is $\lz$-pointed. Then the small box-product $\cbox_{n\in\w}X_n$ is homeomorphic to (an open subset of) the small box-product $\cbox_{n\in\w}E_n$.
\end{theorem}

\begin{proof} Let
$$L=\{n\in\IN:X_n\mbox{ is $l$-pointed}\}\mbox{ \ and \ }Z=\{n\in\IN:X_n\mbox{ is $z$-pointed}\}.$$
 Assume that for
every $n\in\w$ the finite product $\prod_{i\le n}X_n$ is homeomorphic to an open subspace of some typical model space $E_n$. Then each space $X_n$ is metrizable, so its topology is generated by some metrizable uniformity $\U_{X_n}$.

\begin{claim}\label{cl6.2} For every $n\in\IN$ the product $E_{n-1}\times E_n$ is homeomorphic to $E_{n}$.
\end{claim}

\begin{proof} By our assumption, the product $\prod_{i\le n}X_i$ is locally homeomorphic to the model space $E_{n}$. Consequently, there are non-empty open sets $U\subset \prod_{i<n}X_i$ and $V\subset X_n$ whose product $U\times V$ is homeomorphic to an open subset of $E_n$. Since $\prod_{i<n}X_i$ is locally homeomorphic to $E_{n-1}$, the open set $U$ contains a nonempty open set $W$ that is homeomorphic to an open subset of the model space $E_{n-1}$. Since each non-empty open set of $E_{n-1}$ contains an open subset homeomorphic to $E_{n-1}$ we lose no generality assuming that the set $W$ is homeomorphic to $E_{n-1}$. Then $W\times V$ is homeomorphic to an open subset of $E_n$ and hence $E_{n-1}$ is homeomorphic to the retract $W$ of the $E_n$-manifold $W\times V$. By the conditions (3) and (7) of Definition~\ref{d:tms}, the product $E_{n-1}\times E_n$, being a contractible $E_n$-manifold, is homeomorphic to $E_n$.
\end{proof}

\begin{claim}\label{cl6.3} The small box-product $\cbox_{n\in\w}X_n$ is homeomorphic to $\cbox_{n\in\w}X_n\times \cbox_{n\in Z}E_n$.
\end{claim}

\begin{proof} Let $Y_0=X_0$ and for every $n\in\IN$ let $$Y_n=\begin{cases} X_n&\mbox{if $n\notin Z$},\\
X_n\rtimes E_n&\mbox{if $n\in Z$}.
\end{cases}$$
By Proposition~5.5 of \cite{BR10}, the small box products $\cbox_{n\in\w}X_n$ and $\cbox_{n\in\w}Y_n$ can be identified with the uniform direct limits of the towers $\big(\cbox_{i<n}X_i\big)_{n\in\w}$ and  $\big(\cbox_{i<n}Y_i\big)_{n\in\w}$, respectively.

For every $n\in Z$ let $X_n^\circ=X_n\setminus\{*_{X_n}\}$.
By our assumption, the finite product $\cbox_{i\le n}X_i$ is homeomorphic to an open subset of the typical model space $E_n$. Then the space $X_n^\circ\times\cbox_{i<n}X_i$, being an open subset of $\cbox_{i\le n}X_i$ also is homeomorphic to an open subset of $E_n$. Since $E_n$ is a typical model space, the projection
$$\pr:E_n\times X_n^\circ\times \cbox_{i<n}X_i\to X_n^\circ\times \cbox_{i<n}X_i,\;\;\pr:(e,x,\vec x)\mapsto (x,\vec x),$$is a near homeomorphism.

Let $h_0:X_0\to Y_0$ be the identity homeomorphism. Using Lemma~\ref{l:rp},  by induction we can construct a sequence of regular homeomorphisms of pairs $$h_n:(\cbox_{i\le n}X_i,\cbox_{i<n}X_i)\to
(\cbox_{i\le n}Y_i,\cbox_{i<n}Y_i)$$such that $h_n|\cbox_{i<n}X_i=h_{n-1}$. By Corollary~\ref{c2.1}, the map $h:\cbox_{n\in\w}X_n\to \cbox_{n\in\IN}Y_n$ defined by $h|\cbox_{i\le n}X_i=h_n$ is a homeomorphism.

By the condition (8) of Definition~\ref{d:tms}, for every $n\in Z$ the reduced product $X_n\rtimes E_n$ is homeomorphic to $X_n\times E_n$. Consequently, we get the following chain of homeomorphisms:
$$
\begin{aligned}
\cbox_{n\in\w}X_n&\cong\cbox_{n\in\w}Y_n=\cbox_{n\in\w\setminus Z}Y_n\times\cbox_{n\in Z}Y_n=\cbox_{n\in\w\setminus Z}X_n\times\cbox_{n\in Z}(X_n\rtimes E_n)\cong\\
&\cong \cbox_{n\in\w\setminus Z}X_n\times \cbox_{n\in\w}(X_n\times E_n)=\cbox_{n\in\w}X_n\times\cbox_{n\in Z}E_n.
\end{aligned}$$
\end{proof}

\begin{claim}\label{cl6.4} If the set $Z$ is infinite, then the small box product $\cbox_{n\in\w}X_n$ is homeomorphic to an open subspace of the small box-product $\cbox_{n\in \w}E_n$.
\end{claim}

\begin{proof} Let $Z=\{n_k:k\in\w\}$ be the increasing enumeration of the infinite set $Z$. It will be convenient to assume that $n_{-1}=-1$.
For every $k\in\w$ let $Y_k=\prod\limits_{n_{k-1}<i\le n_k}X_{i}$. By Claim~\ref{cl6.3}, the small box-product $\cbox_{n\in\w}X_n$ is homeomorphic to $\cbox_{k\in\w}(Y_{k}\times E_{n_k})$. Since the finite product $\prod\limits_{i\le n_k}X_i$ is homeomorphic to an open subset of $E_{n_{k}}$, the space  $Y_{k}$ is a retract of an open subset of the typical model space $E_{n_k}$ and hence  the product $Y_k\times E_{n_k}$ is homeomorphic to an open subset $U_{n_k}$ of $E_{n_k}$. If the space $Y_k$ is contractible, then $U_{n_k}$, being a contractible $E_{n_k}$-manifold, is homeomorphic to $E_{n_k}$. In this case we can assume that $U_{n_k}=E_{n_k}$.
Claim~\ref{cl6.2} implies that the product $\prod\limits_{n_{k-1}<i\le n_k}E_{i}$ is homeomorphic to $E_{n_k}$ and hence the open set $U_{n_k}$ is homeomorphic to some open set $W_{n_k}$ in $\prod_{n_{k-1}<i\le n_k}E_{i}$ (which coincides with $\prod_{n_{k-1}<i\le n_k}E_i$ if $U_{n_k}=E_{n_k}$).
The space $Y_k\times E_{n_k}$, being  homeomorphic to an open subset $U_{n_k}$ of $E_{n_k}$, is homeomorphic to the open subset $W_{n_k}$ of the product $\prod\limits_{n_{k-1}<i\le n_k}E_i$.

Now we see that the small box product $\cbox_{n\in\w}X_n$ is homeomorphic to the small box-product $\cbox_{k\in\w}(Y_k\times E_{n_k})$ and the latter small box-product is homeomorphic to the small box-product $W=\cbox_{k\in\w}W_{n_k}$, which is an open subset of the small box-product $\cbox_{n\in\w}E_n$. This finishes the proof of Claim~\ref{cl6.4}.
\end{proof}

If each finite product $\prod_{i\le n}X_n$ is homeomorphic to $E_n$, then each space $X_i$, $i\in\w$, is contractible and so are the spaces $Y_k$, $k\in\w$. In this case $U_k=E_{n_k}$ and $W=\cbox_{n\in\w}E_n$. Therefore we have proved the following modification of Claim~\ref{cl6.4}.

\begin{claim}\label{cl6.5} If the set $Z$ is infinite, and each finite product $\prod_{i\le n}X_i$ is homeomorphic to the model space $E_n$, then the small box product $\cbox_{n\in\w}X_n$ is homeomorphic to  $\cbox_{n\in \w}E_n$.
\end{claim}

Claims~\ref{cl6.4} and \ref{cl6.5} prove Theorem~\ref{t:tsbox} in case of infinite set $Z$. If the set $L$ is infinite, then Theorem~\ref{t:tsbox} follows from Claims~\ref{cl6.7} and \ref{cl6.8} proved below.

\begin{claim}\label{cl6.6} If the set $L$ is infinite and each space $X_n$, $n\in\w$, is connected, then the small box product $\cbox_{n\in\w}X_n$ is homeomorphic to an open subspace of  $\cbox_{n\in \w}E_n$.
\end{claim}

\begin{proof} Let $A\subset L$ be an infinite subset such that for each $n\in A$ we get $0<n-1\notin A$. This condition implies that the complement $\IN\setminus A$ is infinite.

By Theorem~\ref{t:sboxlc}, the small box-product $Q\times\cbox_{n\in A}X_n$ is homeomorphic to $K\times Q\times\IR^\infty$ for some connected locally compact polyhedron $K$. If each space $X_n$ is contractible, then we can assume that $K$ is a singleton.
Theorem~\ref{t:sboxlc} also implies that the space $Q\times\IR^\infty$ is homeomorphic to the small box-product $Q\times \cbox_{n\in\w}\II$ where $\II=[0,1]$ is the closed interval with the distinguished point $0$.

First we show that for every $n\in\w$ the product $K\times E_n$ is homeomorphic to an open subset of $E_n$. Indeed, by the condition (1), (3) and (7) of Definition~\ref{d:tms}, the Hilbert cube $Q$ is a retract of the typical model space $E_n$ and the product $E_n\times Q\times[0,1]$ is homeomorphic to $E_n$.
The locally compact polyhedron $K$ is connected and hence admits a closed embedding into $Q\times[0,1)$.
Then $K$ is a neighborhood retract of the space $E_n\times Q\times[0,1)$, which is homeomorphic to an open subspace of $E_n$. By the condition (7) of Definition~\ref{d:tms}, $E_n\times K$ is homeomorphic to an open subset of $E_n$. By Definition~\ref{d:tms}(7), for every open subset $U\subset E_n$ the product $U\times E_n$ is homeomorphic to $U$ and hence $U\times K$ is homeomorphic to $U\times E_n\times K$ and the latter space is homeomorphic to an open subset of the square $E_n\times E_n$, which is homeomorphic to $E_n$ (being a contractible $E_n$-manifold).

Since $X_0$ is homeomorphic to an open subset of $E_0$, the product $X_0\times K$ is homeomorphic to an open subset of the model space $E_0$.
Since $X_0$ is homeomorphic to $X_0\times E_0$ and $E_0$ is homeomorphic to $E_0\times Q$, the space $X_0$ is homeomorphic to $X_0\times Q$. So, we get the following chain of homeomorphisms
$$
\begin{aligned}
\cbox_{n\in\w}X_n&\cong X_0\times Q\times\cbox_{n\in\IN}X_n\cong X_0\times \cbox_{n\in\IN\setminus A}X_n\times (Q\times\cbox_{n\in A}X_n)\cong\\
&\cong X_0\times \cbox_{n\in\IN\setminus A}X_n\times (K\times Q\times \cbox_{n\in\w}\II)\cong X_0\times K\times Q\times\cbox_{n\in\IN\setminus A}(X_n\times \II)\cong\\
&\cong X_0\times K\times\cbox_{n\in\IN\setminus A}(X_n\times \II).
\end{aligned}
$$
By Lemma~\ref{l:Z}, for every $n\in\IN\setminus A$ the distinguished point $(*_{X_n},0)$ of the pointed space $X_n\times \II$ is a strong $Z$-point.

 Then by Claim~\ref{cl6.4}, the small box-product $\cbox_{n\in\w}X_n\cong K\times X_0\times \cbox_{n\in\IN\setminus A}(X_n\times \II)$ is homeomorphic to an open subset of $\cbox_{n\in\w\setminus A}E_n$. By Claim~\ref{cl6.2}, for every $n\in\IN\setminus A$ the space $E_n$ is homeomorphic to $E_n\times E_{n-1}$.
Consequently, the small box-product $\cbox_{n\in\w\setminus A}E_n$ is homeomorphic to $\cbox_{n\in\w}E_n$ and thus $\cbox_{n\in\w}X_n$ is homeomorphic to an open subspace of $\cbox_{n\in\w}E_n$.
\end{proof}

By analogy we can prove:

\begin{claim}\label{cl6.7} If the set $L$ is infinite, and every finite product $\prod_{i\le n}X_n$ is contractible, then the small box product $\cbox_{n\in\w}X_n$ is homeomorphic to  $\cbox_{n\in \w}E_n$.
\end{claim}

Our final claim finishes the proof of Theorem~\ref{t:tsbox}.

\begin{claim}\label{cl6.8} If the set $L$ is infinite, then the small box product $\cbox_{n\in\w}X_n$ is homeomorphic to an open subspace of the small box-product $\cbox_{n\in \w}E_n$.
\end{claim}

\begin{proof} For every $n\in\w$ denote by $\kappa_n$ the number of connected components of the space $X_n$. It is clear that the finite product $\prod_{i\le n}X_i$ has $\prod_{i\le n}\kappa_i$ many connected components. Since $\prod_{i\le n}X_i$ is homeomorphic to an open subset of $E_n$, the model space $E_n$ contains a family $\U_n$ consisting of $\prod_{i\le n}\kappa_i$ many pairwise disjoint non-empty open subsets. Since each non-empty open subset of $E_n$ contains an open subset homeomorphic to $E_n$, we can assume that each set $U\in\U_n$ is homeomorphic to $E_n$. Since $E_n$ is topologically homogeneous, we can assume that its distinguished point lies in the union $U_n=\cup\U_n$. It follows that $U=\cbox_{n\in\w}U_n$ is an open subset of $\cbox_{n\in\w}E_n$ and each connected component of $U$ is homeomorphic to $\cbox_{n\in\w}E_n$. Observe that the spaces $\cbox_{n\in\w}X_n$ and $\cbox_{n\in\w}U_n$ consist of $\kappa=\sup_{n\in\w}\prod_{i\le n}\kappa_i$ many connected components. So, we can choose a bijective map $\gamma$ assigning to each connected component of $\cbox_{n\in\w}X_n$ a connected component of the space $\cbox_{n\in\w}U_n$. By Claim~\ref{cl6.6}, each connected component $C$ of $\cbox_{n\in\w}X_n$ is homeomorphic to an open subset of $\cbox_{n\in\w}E_n$. So, we can define an open topological embedding $f_C:C\to\gamma(C)$ of $C$ into the connected component $\gamma(C)$ of the small box-product $\cbox_{n\in\w}U_n$. Then the union $f=\bigcup_Cf_C:\cbox_{n\in\w}X_n\to\cbox_{n\in\w}U_n\subset\cbox_{n\in\w}E_n$ is a required open embedding of $\cbox_{n\in\w}X_n$ into $\cbox_{n\in\w}E_n$.
\end{proof}
\end{proof}

\section{Recognizing the topology of some uniform direct limits}\label{s:tudl}

In this section we prove a ``typical'' version of Theorem~\ref{t:udl}.

\begin{theorem}\label{t:tudl} Let $(E_n)_{n\in\w}$ be a sequence of typical model spaces. The uniform direct limit $\ulim X_n$ of a tower of uniform spaces $(X_n)_{n\in\w}$ is
\begin{enumerate}
\item homeomorphic to (an open subset of) $\cbox_{n\in\w}E_n$ if each space $X_n$ is $\lz$-complemented in $X_{n+1}$ and $X_n$ is homeomorphic to (an open subset of) the model space $E_n$;
\item (locally) homeomorphic to $\cbox_{n\in\w}E_n$ if each space $X_n$ (locally) $\lz$-complemented in $X_{n+1}$ and $X_n$ is (locally) homeomorphic to $E_n$;
\item homeomorphic to $\cbox_{n\in\w}E_n$ if each uniform space $X_n$ is metrizable, homeomorphic to $E_n$ and is locally $z$-complemented in $X_{n+1}$.
\end{enumerate}
\end{theorem}

\begin{proof} 1. Assume that for every $n\in\w$ the space $X_n$ is $\lz$-complemented in $X_{n+1}$ and is homeomorphic to (an open subset of) the model space $E_{n+1}$. Then $X_n$ is $C_n$-complemented in $X_{n+1}$ for some $\lz$-pointed space $C_n$.
By Theorem~\ref{t:udl-cbox}, the uniform direct limit $\ulim X_n$ is homeomorphic to the small box-product $X_0\times\cbox_{n\in\w}C_n$.

The $C_n$-complementedness of $X_n$ in $X_{n+1}$ implies that the space $X_{n+1}$ is homeomorphic to the product $X_n\times C_n$. Continuing by induction we can prove that $X_{n+1}$ is homeomorphic to $X_0\times\prod_{i\le n}C_n$. Then the latter product is homeomorphic to (an open  subset of) $E_{n+1}$ and we can apply Theorem~\ref{t:tsbox} to prove that $X_0\times \cbox_{n\in\w}X_n$ is homeomorphic to (an open subset of) the small box-product $\cbox_{n\in\w}E_n$.
\smallskip

2. Now assume that for every for every $n\in\w$ the space $X_n$ is locally $\lz$-complemented in $X_{n+1}$ and is locally homeomorphic to $E_n$. We need to show that each point $x\in\ulim X_n$ has an open neighborhood homeomorphic to an open subset of $\cbox_{n\in\w}E_n$.
Find a number $n\in\w$ with $x\in X_n$. By Lemma~\ref{l1n}, the point $x$ has an open neighborhood which is homeomorphic to the small box-product $X_n\times\cbox_{i\ge n}W_i$ for some open neighborhoods $W_i\subset C_i$ of the distinguished points $*_{C_i}$. It follows from the proof of Lemma~\ref{l1n} that for every $m\ge n$ the product $X_n\times \prod_{n\le i\le m}W_i$ is homeomorphic to an open subset of $X_{n+1}$ and hence is locally homeomorphic to $E_{n+1}$. Then Theorem~\ref{t:tsbox} guarantees that the small box-product $X_n\times\cbox_{i\ge n}W_i$ is locally homeomorphic to $\cbox_{i\ge n}E_i$. Consequently, the point $x$ has a neighborhood $O(x)$ that is homeomorphic to an open subset of $\cbox_{i\ge n}E_i$.

Repeating the argument of Claim~\ref{cl6.2}, we can prove that for every $k\in\IN$ the model space $E_k$ is homeomorphic to $E_{k-1}\times E_k$. Then $E_{n-1}$ is homeomorphic to $\prod_{i<n}E_i$ and $\cbox_{i\ge n}E_i$ is homeomorphic to $\cbox_{n\in\w}E_n$. Now we see that $O(x_0)$ is homeomorphic to an open subset of $\cbox_{n\in\w}E_n$ and hence the uniform direct limit $\ulim X_n$ is locally homeomorphic to $\cbox_{n\in\w}E_n$.
\smallskip

3. Assume that each uniform space $X_n$ is metrizable, homeomorphic to $E_n$ and is locally $z$-complemented in $X_{n+1}$. By Lemma~\ref{l:zcomp}, the set $X_n$ is $E_{n+1}$-complemented in $X_{n+1}$ and by the statement (1), $X=\ulim X_n$ is homeomorphic to the small box-product $\cbox_{n\in\w}E_n$.
\end{proof}

\section{Proof of Theorem~\ref{t:sbox}}\label{s:pfsbox}

Let $(X_n)_{n\in\w}$ be a sequence of pointed topological spaces such that for every $n\in\w$ the finite product $\prod_{i\le n}X_i$ is homeomorphic to (an open subset of) a Hilbert space $E_n$ and for infinitely many numbers $n\in\w$ the space $X_n$ is $\lz$-pointed.

We consider three cases.
\smallskip

I. For some $m\in\w$ the Hilbert space $E_m$ is infinite-dimensional.
Then for all $n\ge m$ the Hilbert spaces $E_n$ are infinite-dimensional. In this case we can apply Theorem~\ref{t:tsbox} to conclude that the small-box product $\cbox_{n>m}X_n$ is homeomorphic to (an open subset of)  $\cbox_{n>m}E_n$. Since the product $\prod_{i\le m}X_i$
is homeomorphic to (an open subset of) the Hilbert space $E_{m}$, the small box-product $\cbox_{n\in\w}X_n$ is homeomorphic to (an open subset of) the small box-product $\cbox_{n\ge m}E_n$ of Hilbert spaces. The latter box product can be identified with the LF-space $\bigoplus\limits_{n\ge m}E_n$.
\smallskip

II. All Hilbert spaces $E_n$ are finite-dimensional and $\sup_{n\in\w}\dim(E_n)=\infty$. In this case for every $n\in\w$ the finite product $\prod_{i\le n}X_n$, being homeomorphic to an open subset of the finite-dimensional Hilbert space $E_n$, is a locally compact $\sigma$-compact finite-dimensional ANR.
Then the topological direct limit $\tlim \cbox_{i\le m}X_i$ of the tower $\big(\cbox_{i\le n}X_i\big)_{n\in\w}$ can be written as the topological direct limit of a tower of finite-dimensional metrizable compacta.
By Propositions 5.4 and 5.5 of \cite{BR10}, the identity map
$$\tlim \cbox_{i\le n}X_n\to \cbox_{n\in\w}X_n$$ is a homeomorphism.
Now by Theorem~\ref{sakai}, the topological equivalence of $\cbox_{n\in\w}X_n$ to (an open subset of) the LF-space $\IR^\infty$ will follow as soon as we prove that each embedding $f:B\to \cbox_{n\in\w}X_n$ of a closed subset $B$ of a finite-dimensional metrizable compact space $A$ can be extended to an embedding $\bar f$ of (some neighborhood of $B$ in) the space $A$.

Since $f(B)$ is a compact subset of the topological direct limit $\cbox_{i\in\w}X_i=\tlim \cbox_{i\le n}X_i$, there is $n\in\w$ such that $f(B)\subset\cbox_{i\le n}X_i$. Since $\cbox_{i\le n}X_i$ is an ANR-space, the map $f$ admits a continuous extension $\bar f:O(B)\to \cbox_{i\le n}X_i$ defined on a closed neighborhood $O(B)$ in $A$. Since $\cbox_{i\le n}X_i$ is homeomorphic to the Hilbert space $E_n$, then it is an absolute retract and we can additionally assume that $O(B)=A$.

Now consider the quotient space $O(B)/B$ and the corresponding quotient map $\pi:O(B)\to O(B)/B$. Being metrizable and finite-dimensional, the compact space $O(B)/B$ admits an embedding $e:O(B)/B\to \II^k$ for some $k$ such that the distinguished point $B$ of $O(B)/B$ maps onto the distinguished point $(0,\dots,0)$ of the cube $\II^k$. Then the map
$$\tilde f:O(B)\to \II^k\times \prod_{i\le n}X_i,\;\;\tilde f:x\mapsto (e\circ\pi(x),\bar f(x))$$ is a topological embedding.

Consider the set $M$ of all numbers $m$ for which the point $*_{X_m}$ is not isolated in the ANR-space $X_n$. It follows from $\sup_{m\in\w}\dim(E_m)=\infty$ that the set $N$ is infinite. For every $m\in M$ we can find an embedding $\gamma_m:\II\to X_m$ such that $\gamma_m(0)=*_{X_m}$. Since the set $M$ is infinite, we can choose a sequence of numbers $m_1<m_2<\dots<m_k$ in $M$ such that $m_1>n$. Consider the embedding $\gamma:\II^k\to\prod_{i=1}^k X_{m_k}$ defined by $\gamma:(t_1,\dots,t_k)\mapsto (\gamma_{m_1}(t_1),\dots,\gamma_{m_k}(t_k))$.

Identify the product $\prod_{i=1}^kX_{m_i}$ with the subset
$$\{(x_i)_{i=n+1}^{m_k}\in\prod_{i=n+1}^{m_k}X_i: i\notin\{m_1,\dots,m_k\}\;\Ra\;x_i=*_{X_i}\}$$of the product $\prod_{i=n+1}^{m_k}X_i$ and consider the embedding
$$\delta:\Big(\prod_{i\le n}X_i\Big)\times\II^k\to \prod_{i\le m_k}X_i=\prod_{i\le n}X_i\times\prod_{i=n+1}^{m_k}X_i,\;\;\delta:(x,t)\mapsto \big(x,\gamma(t)\big)
.$$ Then the composition $\delta\circ \tilde f:O(B)\to \prod_{i\le m_k}X_i\subset\cbox_{i\in\w}X_i$ is a required embedding of $O(B)$ that extends the embedding $f$. Now it is legal to apply Theorem~\ref{sakai} and conclude that the small box-product $\cbox_{i\in\w}X_i$ is homeomorphic to (an open subspace of) the LF-space $\IR^\infty$.
\smallskip

III. $k=\sup_{n\in\w}\dim(E_n)$ is finite. Then there is $m\in\w$ such that $\dim(E_n)=k$ for all $n\ge m$. For every $n<m$ the finite products $\cbox_{i<n}X_i$ and $\cbox_{i\le n}X_i$ are homeomorphic to open subsets of the Euclidean space $\IR^k$. By the Brouwer Domain Preservation Principle \cite{Brouwer}, the space $\cbox_{i<n}X_i$ is open in $\cbox_{i\le n}X_i$. This implies that the space $X_n$ is discrete and at most countable. Then the small box-product $\cbox_{n>m}X_n$ is discrete and at most countable. Since the product $\cbox_{i\le m}X_i$ is homeomorphic to an open subset of $\IR^k$ and $\IR^k$ contains an open subspace homeomorphic to $\IR^k\times\w$, the small box-product $\cbox_{i\in\w}X_i=\cbox_{i\le m}X_i\times\cbox_{i>m}X_i$ is homeomorphic to an open subset of $E_m$.

If each finite product $\prod_{i\le n}X_i$ is homeomorphic to the Hilbert space $E_n$, then for every $n>m$ the space $X_n$ is a singleton. Consequently, the small box-product $\cbox_{n\in\w}X_i=\cbox_{i\le m}X_i$ is homeomorphic to the finite-dimensional LF-space $E_m$.

\section{Acknowledgment}

The authors would like to express their thanks to Katsuro Sakai and to the anonymous referee (who noticed a gap in the original version of the paper) for valuable suggestions and remarks.


\begin{thebibliography}{MM}

\bibitem{Ba98a} T.~Banakh, {\em On linear topological spaces (linearly) homeomorphic to $\IR^\infty$}, Mat. Stud. {\bf 9}:1 (1998) 99--101.

\bibitem{Ba98} T.~Banakh, {\em On topological groups containing a Fr\'echet-Urysohn fan}, Mat. Stud. {\bf 9}:2 (1998), 149--154.

\bibitem{BMRSY} T.~Banakh, K.~Mine, D.~Repov\v s, K.~Sakai, T.~Yagasaki, {\em On topological groups (locally) homeomorphic to LF-spaces}, preprint (arXiv:1004.0305).

\bibitem{BMSY} T.~Banakh, K.~Mine, K.~Sakai, T.~Yagasaki, {\em On homeomorphism groups of non-compact surfaces, endowed with the Whitney topology}, preprint (arXiv:1004.3015).

\bibitem{BRZ} T.Banakh, T.Radul, M.Zarichnyi, Absorbing Sets in Infinite-Dimensional Manifolds, VNTL Publishers, Lviv, 1996. 240p.

\bibitem{BR10} T.~Banakh, D.~Repov\v s, {\em The topological structure of direct limits in the category of uniform spaces}, Topology Appl. 157 (2010) 1091-1100.

\bibitem{BY} T.~Banakh, T.~Yagasaki,
{\em Diffeomorphism groups of non-compact manifolds endowed with the Whitney $C^\infty$-topology}, preprint (arXiv: 1005.1789).


\bibitem{BZ} T.~Banakh, L.~Zdomskyy, {\em The topological structure of (homogeneous) spaces and groups with countable cs*-network}, Appl. Gen. Top. {\bf 5}:1 (2004), 25--48.

\bibitem{BP} C.Bessaga, A.Pe\l czy\'nski, Selected Topics in Infinite-Dimensional Topology, PWN, Warsaw, 1975.

\bibitem{BM} M.Bestvina, J.Mogilski, {\em Characterizing certain incomplete infinite-dimensional absolute retracts}, Michigan Math. J. {\bf 33}:3 (1986) 291--313.

\bibitem{BBMW} M.~Bestvina, P.~Bowers, J.~Mogilski, J.~Walsh, {\em Characterization of Hilbert space manifolds revisited}, Topology Appl. {\bf 24}:1--3 (1986), 53--69.

\bibitem{Brouwer} L.E.~Brouwer, {\em Beweis der Invarianz der geschlossenen Kurve}, Math. Ann. {\bf 72}:3 (1912), 422--425 (in German).

\bibitem{Chap} T.A.~Chapman, Lectures on Hilbert cube manifolds, Amer.  Math. Soc., Providence, R. I., 1976.

\bibitem{Chi} A.~Chigogidze, {\em Inverse Spectra}, North-Holland Publishing Co., Amsterdam, 1996.

\bibitem{En} R.~Engelking, {General Topology}, Heldermann Verlag, Berlin, 1989.


\bibitem{HS} D.~Henderson, R.~Schori, {\em Topological classification of infinite-dimensional manifolds by homotopy type}, Bull. Amer. Math. Soc. {\bf 76} (1970), 121--124.

\bibitem{Man} P.~Mankiewicz, {\em On topological, Lipschitz, and uniform classification of $LF$-spaces}, Studia Math. {\bf 52} (1974) 109--142.

\bibitem{MS} K.~Mine, K.~Sakai, {\em Open subsets of LF-spaces}, Bull. Pol. Acad. Sci. Math. {\bf 56}:1 (2008), 25--37.

\bibitem{MS2} K.~Mine, K.~Sakai, {\em Simplicial complexes and open subsets of non-separable LF-spaces}, Canadian J. Math. {\bf 63}:2 (2011) 436--459.

\bibitem{Pen} E.~Pentsak, {\em On manifolds modeled on direct limits of $\C$-universal ANR's}, Mat. Stud. {\bf 5} (1995) 107--116.

\bibitem{Sak84} K.~Sakai, {\em On $\mathbb R^{\infty }$-manifolds and $Q\sp{\infty }$-manifolds}, Topology Appl. {\bf 18}:1 (1984) 69--79.


\bibitem{Tor74} H.~Toru\'nczyk, {\em Absolute retracts as factors of normed linear spaces}, Fund. Math. {\bf 86} (1974) 53--67.

\bibitem{Tor81} H.~Toru\'nczyk, {\em Characterizing Hilbert space topology}, Fund. Math. {\bf 111}:3 (1981) 247--262.

\bibitem{Tor85} H.~Toru\'nczyk, {\em A correction of two papers concerning Hilbert manifolds}, Fund. Math. {\bf 125}:1 (1985) 89--93.
\end{thebibliography}
\end{document}